\documentclass[oneside]{amsart}

\usepackage[a4paper]{geometry}
\usepackage{amssymb,amsfonts,amsmath,stmaryrd}
\usepackage{comment}
\usepackage[all,arc]{xy}
\usepackage{enumerate}
\usepackage{mathrsfs}
\usepackage{tikz}
\usepackage{tikz-cd}
\usetikzlibrary{trees}
\usetikzlibrary[shapes]
\usetikzlibrary[arrows]
\usetikzlibrary{patterns}
\usetikzlibrary{fadings}
\usetikzlibrary{backgrounds}
\usetikzlibrary{decorations.pathreplacing}
\usetikzlibrary{decorations.pathmorphing}
\usetikzlibrary{positioning}
\usepackage{xcolor} %improved colors
\usepackage{graphicx}
\usepackage[pagebackref]{hyperref}

\definecolor{dark-red}{rgb}{0.5,0.15,0.15}
\definecolor{dark-blue}{rgb}{0.15,0.15,0.6}
\definecolor{dark-green}{rgb}{0.15,0.6,0.15}
\hypersetup{
    colorlinks, linkcolor=dark-red,
    citecolor=dark-blue, urlcolor=dark-green
}

\renewcommand*{\backref}[1]{}
\renewcommand*{\backrefalt}[4]{%
  \ifcase #1 %
No citations.% use \relax if you do not want the "No citations" message
  \or
(cit. on p. #2).%
  \else
(cit on pp. #2).%
  \fi%
}

\usepackage[nameinlink,capitalise,noabbrev]{cleveref}
\newtheorem{thm}{Theorem}[section]
\newtheorem{cor}[thm]{Corollary}
\newtheorem{prop}[thm]{Proposition}

\theoremstyle{definition}
\newtheorem{defn}[thm]{Definition}

\newtheorem{ex}[thm]{Example}

\theoremstyle{remark}
\newtheorem{rem}[thm]{Remark}
\theoremstyle{theorem}
\newtheorem*{thm*}{Theorem}
\newtheorem*{thma}{Theorem A}
\newtheorem*{thmb}{Theorem B}
\newtheorem*{freydgh}{Freyd's generating hypothesis}
\newtheorem*{gh}{The generating hypothesis}
\newtheorem*{kngh}{The $K(n)$-local generating hypothesis}
\newtheorem*{pickngh}{The $\Pic$-graded $K(n)$-local generating hypothesis}
\theoremstyle{remark}
\newtheorem*{rem*}{Remark}

\makeatletter
\let\c@equation\c@thm
\makeatother
\numberwithin{equation}{section}

\DeclareMathOperator{\Sp}{Sp}
\DeclareMathOperator{\Hom}{Hom}
\DeclareMathOperator{\End}{End}

\DeclareMathOperator{\cC}{\mathcal{C}}
\DeclareMathOperator{\cD}{\mathcal{D}}

\DeclareMathOperator{\Mod}{Mod}

\DeclareMathOperator{\Thick}{Thick}

\DeclareMathOperator{\op}{op}
\DeclareMathOperator{\dual}{dual}

\DeclareMathOperator{\Pic}{Pic}
\DeclareMathOperator{\Perf}{Perf}
\DeclareMathOperator{\Coh}{Coh}
\DeclareMathOperator{\StMod}{StMod}

\newcommand{\Z}{\mathbb{Z}}
\newcommand{\Q}{\mathbb{Q}}

\Crefname{figure}{Figure}{Figures}
\Crefname{assu}{Assumption}{Assumptions}

\newcommand{\cF}{\mathcal{F}}
\newcommand{\cG}{\mathcal{G}}
\newcommand{\cO}{\mathcal{O}}

\title[The $K(n)$-local generating hypothesis]{Auslander--Reiten sequences, Brown--Comenetz duality, and the $K(n)$-local generating hypothesis}

\author{Tobias Barthel}
\address{Max Planck Institute for Mathematics, Bonn, Germany}
\email{tbarthel@mpim-bonn.mpg.de}

\date{\today}

\begin{document}
\maketitle

\begin{abstract}
In this paper, we construct a version of Auslander--Reiten sequences for the $K(n)$-local stable homotopy category. In particular, the role of the Auslander--Reiten translation is played by the local Brown--Comenetz duality functor. As an application, we produce counterexamples to the $K(n)$-local generating hypothesis for all heights $n>0$ and all primes. Furthermore, our methods apply to other triangulated categories, as for example the derived category of quasi-coherent sheaves on a smooth projective scheme. 
\end{abstract}

%\tableofcontents

\section{Introduction}

In his paper \cite{freydgh}, Freyd proposed the following generating hypothesis for the homotopy category of finite spectra $\Sp^{\omega}$, arguably one of the most important open problems in stable homotopy theory. 

\begin{freydgh}
A map $f\colon X \to Y$ between finite spectra $X$ and $Y$ is null-homotopic if and only if $\pi_*f = 0$.
\end{freydgh}

Note that the finiteness assumption is crucial, as the generating hypothesis fails in the whole stable homotopy category: as a counterexample, one may take any positive degree Steenrod operation. Surprisingly, Freyd showed that faithfulness of $\pi_*$ would imply fullness, thereby producing an embedding
\[
\xymatrix{\pi_*\colon \Sp^{\omega} \ar[r] & \Mod_{\pi_*S^0}.}
\]
However, the generating hypothesis for $\Sp^{\omega}$ remains an open problem, and besides a number of interesting consequences that could be deduced~\cite{freydgh, hoveygh}, essentially nothing is known about it. 

To gain some intuition about Freyd's original statement, it is natural to seek to understand versions of the generating hypothesis in structurally similar categories. Examples include the derived category of a ring \cite{hlpgenhypdermod} and the stable module category appearing in modular representation theory \cite{ccmgenhypstmod}; a more abstract approach is taken in \cite{bohmannmaygh}.

In a different direction, Devinatz and Hopkins \cite{devinatzgh1,devinatzgh2} describe a programme to prove the generating hypothesis with target $S^0$ by reducing it to a family of conjectures in chromatic stable homotopy theory. Motivated by their approach, we formulate and study a $K(n)$-local version of the generating hypothesis, where $K(n)$ denotes Morava $K$-theory at a given height $n$ and prime $p$. Roughly speaking, the $K(n)$-local categories stratify the stable homotopy category and are thus the fundamental building blocks of the stable homotopy category. In analogy with Freyd's original statement, the $K(n)$-local generating hypothesis asserts that $\pi_*$ defines a faithful functor on the category of compact $K(n)$-local spectra. 

\begin{kngh}
If $f\colon X \to Y$ is a morphism in $\Sp_{K(n)}^{\omega}$ with $\pi_*f = 0$, then $f=0$. 
\end{kngh}

This statement is related to conjectures formulated by Devinatz in \cite{devinatzgh2}, although they appear to be logically independent from the version of the $K(n)$-local generating hypothesis considered here. Therefore, the results of this paper do not have any direct consequences for the Devinatz--Hopkins approach to Freyd's generating hypothesis with target $S^0$. 

Our approach to the $K(n)$-local generating hypothesis is inspired by the solution \cite{ccmgenhypstmod} of the generating hypothesis for the stable module category of a finite group $G$ over a field $k$ of characteristic $p$ dividing the order of $G$. In this case, the existence of Auslander--Reiten sequences (or almost split sequences) allows to construct counterexamples to the generating hypothesis whenever the thick subcategory generated by $k$ is not semi-simple. Observing that Auslander--Reiten sequences have been constructed in very general settings \cite{krauseartbrown,beligiannisart}, this offers an approach to the generating hypothesis for general triangulated categories. 

Suppose $\cC$ is a suitable triangulated category with a distinguished object $S$, and define the homotopy groups of any $X \in \cC$ as $\pi_*X = \Hom_{\cC}^*(S,X)$. We say that the generating hypothesis holds for a full subcategory $\cD \subseteq \cC$ if the restriction of $\pi_*$ to $\cD$ is faithful. Our first theorem gives sufficient conditions for the existence of counterexamples to the generating hypothesis in $\cD$. 

\begin{thma}
Suppose $\cC$ is a compactly generated stable homotopy category, and let $\cD \subseteq \cC$ be a thick subcategory that admits Auslander--Reiten sequences. If $\cD$ contains indecomposable objects which are not equivalent to a retract of a direct sum of shifted copies of $S$, then the generating hypothesis does not hold in $\cD$.
\end{thma}

As an application, we show that the generating hypothesis fails for the derived category of quasi-coherent sheaves over a non-trivial smooth projective scheme over a field $k$.

The existence of Auslander--Reiten sequences is closely related to the existence of an Auslander--Reiten translation, which simultaneously generalizes the notion of Serre duality in algebraic geometry and Brown--Comenetz duality in stable homotopy theory. Using a local version of Brown--Comenetz duality, in \Cref{sec:kngh}, we verify that the $K(n)$-local stable homotopy category $\Sp_{K(n)}$ satisfies the conditions of Theorem A, thereby obtaining our second theorem.

\begin{thmb}
The category of compact $K(n)$-local spectra admits Auslander--Reiten sequences, with the Auslander--Reiten translation given by the dual of the local Brown--Comenetz duality functor. In particular, the $K(n)$-local generating hypothesis fails for any $n>0 $ and all primes $p$.
\end{thmb}

Finally, we discuss a Picard group graded analogue of the $K(n)$-local generating hypothesis, with the same conclusion as for the $\Z$-graded one, see \Cref{cor:piclocalgenhyp}. We believe it would be interesting in its own right to study the Auslander--Reiten theory of $\Sp_{K(n)}$, e.g., to try to compute its Auslander--Reiten quiver, but we leave this task to future works. 

\subsection*{Conventions and terminology}

Throughout this paper, we will work in the context of axiomatic stable homotopy theory as developed in \cite{hps_axiomatic}. In particular, $\cC = (\cC,\otimes,S)$ will always denote a closed symmetric monoidal cocomplete triangulated category which is generated by a set of dualizable objects. Here, $S$ is the tensor unit of $\cC$. The suspension functor on $\cC$ is denoted $\Sigma$ and we usually refer to distinguished triangles as cofiber sequences. We write $\Hom_{\cC}(X,Y)$ for the abelian group of maps between two objects $X,Y \in \cC$ and $\Hom_{\cC}^*(X,Y) = \Hom_{\cC}(X,\Sigma^*Y)$ for the corresponding graded abelian group. If $F$ denotes the internal function object in $\cC$, then we define the Spanier--Whitehead dual of an object $X\in \cC$ as $X^{\vee}=F(X,S)$.

The full subcategory of compact objects of $\cC$ will be denoted by $\cC^{\omega}$. For any object $X \in \cC$, we write $\Thick_{\cC}(X)$ for the thick subcategory of $\cC$ generated by $X$, i.e., the smallest full subcategory of $\cC$ which contains $X$ and that is closed under finite colimits, suspensions, and retracts.

\subsection*{Acknowledgements}

I would like to thank Frank Gounelas, Drew Heard, Keir Lockridge, Nathaniel Stapleton, and Geordie Williamson for helpful converations, and the Max Planck Institute for Mathematics for its hospitality.

\section{Auslander--Reiten sequences in stable homotopy categories}\label{sec:abstract}

In this section, we develop an abstract framework in which one can study the generating hypothesis. We start by reviewing the tools from abstract Auslander--Reiten theory that we then use in \Cref{ssec:abstractgh} to prove our criterion for the existence of counterexamples to the generating hypothesis in a compactly generated stable homotopy category.

\subsection{Abstract Auslander--Reiten theory}\label{ssec:abstractartgh}

Auslander--Reiten theory was originally used as a device to analyse the structure of the category of modules over a finite dimensional algebra over a field $k$. The fundamental constructions, however, can be carried out in much greater generality, and we mostly follow \cite{krauseartbrown} and \cite{krauseartzimmermann} in our exposition. A more general and more comprehensive account is given in \cite{beligiannisart}.

\begin{defn}\label{defn:arsequence}
Suppose $\cC$ is a stable homotopy category. A cofiber sequence
\begin{equation}\label{eq:generalarseq}
\xymatrix{\Sigma^{-1}Z \ar[r]^-{f} & X \ar[r]^-{g} & Y}
\end{equation}
in $\cC$ is called an Auslander--Reiten sequence if it satisfies the following three conditions:
\begin{enumerate}
	\item The sequence \eqref{eq:generalarseq} does not split. 
	\item If $g'\colon Y' \to Y$ is not a split epimorphism, then there exists a factorization
	\[
	\xymatrix{& Y' \ar[d]^{g'} \ar@{-->}[dl] \\
	X \ar[r]_-g & Y.}
	\]
	\item Dually, every map $f'\colon \Sigma^{-1}Z \to Z'$ that is not a split monomorphism factors through $f$. 
\end{enumerate}
We say that a thick subcategory $\cD \subseteq \cC$ admits Auslander--Reiten sequences if any indecomposable object $Y \in \cD$ fits into an Auslander--Reiten sequence \eqref{eq:generalarseq} with $X,Z \in \cD$.
\end{defn}

It can be shown that the objects $Y$ and $Z$ appearing in an Auslander--Reiten sequence \eqref{eq:generalarseq} are indecomposable, and that the sequence is determined up to equivalence by either of them. For the details, we refer the interested reader to~\cite[Sec.~2]{krauseartbrown}, which also establishes several other characterizations of Auslander--Reiten sequences. This motivates to ask to what extent $Z$ depends functorially on $Y$. 

In order to answer this question, let us assume from now on that $\cC$ is $k$-linear for some complete local Noetherian commutative ring $k$ with maximal ideal $\mathfrak{m}$. In particular, this means that $\Hom_{\cC}(X,Y)$ is canonically a $k$-module for all objects $X,Y \in\cC$. We denote the injective hull of $k/\mathfrak{m}$ by $E=E(k)$ and consider the Matlis duality functor
\[
\xymatrix{D = \Hom_k(-,E)\colon \Mod_k \ar[r] & \Mod_k}
\]
on the category $\Mod_k$ of $k$-modules. By Brown representability for $\cC$ and using that $E$ is injective, the functor 
\[
\xymatrix{D\Hom_{\cC}(X,-)\colon \cC^{\op} \ar[r] & \Mod_k}
\]
is representable by an object $TX \in \cC$ for any $X \in \cC$. Under suitable conditions on $\cC$, $TX$ is equivalent to $Z$ in the Auslander--Reiten sequence \eqref{eq:generalarseq}, as we will see shortly. 

\begin{defn}\label{defn:artranslation}
An Auslander--Reiten translation for $\cC$ is a fully faithful functor $T\colon \cC^{\omega} \to \cC$ satisfying the Auslander--Reiten formula, i.e., such that there is a binatural equivalence
\begin{equation}\label{eq:arduality}
\xymatrix{D\Hom_{\cC}(X,Y) \cong \Hom_{\cC}(Y,TX)}
\end{equation}
of $k$-modules, for all $X \in \cC^{\omega}$ and all $Y \in \cC$.
\end{defn}

An abstract existence theorem for Auslander--Reiten sequences was proven by Krause in \cite[Prop.~4.2, Thm.~4.4]{krauseartzimmermann}; it is summarized in the following theorem.

\begin{thm}[Krause]\label{thm:art}
Suppose that $\cC$ is a $k$-linear compactly generated stable homotopy category such that $\Hom_{\cC}(X,Y) \in \Mod_k$ is finitely generated for all $X,Y \in \cC^{\omega}$, then $\cC$ admits an Auslander--Reiten translation $T$, and the following statements are equivalent:
\begin{enumerate}
	\item The functor $T\colon \cC^{\omega} \to \cC$ induces an equivalence
	\[
	\xymatrix{\cC^{\omega} \ar[r]^-{T} \ar@{-->}[d]_{\sim} & \cC \\
	\cC^{\omega}, \ar@{^{(}->}[ru]_{\iota}}
	\]
	where $\iota$ is the natural inclusion. 
	\item The subcategory $\cC^{\omega} \subseteq \cC$ admits Auslander--Reiten sequences.
\end{enumerate}
In this case, the Auslander--Reiten sequence for an indecomposable object $Y \in \cC^{\omega}$ has the form
\[
\xymatrix{\Sigma^{-1}TY \ar[r] & X \ar[r] & Y,}
\]
with $X,TY \in \cC^{\omega}$. 
\end{thm}

\begin{rem}
Although none of the results in this paper require a homotopical enhancement, the basic properties of Auslander--Reiten sequences can be established on the level of stable $\infty$-categories, as worked out in the thesis of West~\cite{westart}.
\end{rem}

\subsection{The generating hypothesis}\label{ssec:abstractgh}

Throughout this section, let $\cC$ be a compactly generated stable homotopy category with tensor unit $S$. By definition, the homotopy groups of an object $X \in \cC$ are given by $\pi_*X = \Hom_{\cC}(\Sigma^*S,X)$. This construction extends via postcomposition to an additive functor 
\[
\xymatrix{\pi_*\colon \cC \ar[r] & \Mod_{\pi_*S}.}
\]
Suppose $\cD$ is a full subcategory of $\cC$. The generating hypothesis for $\cD$ asserts that $\pi_*$ is faithful when restricted to $\cD$.

\begin{gh}
Any morphism $f\colon X \to Y$ in $\cD$ with $\pi_*f = 0$ is null. 
\end{gh}

As explained in the introduction, the generating hypothesis was originally proposed in the context of stable homotopy theory by Freyd in~\cite{freydgh}, but has since been studied in other, structurally similar stable homotopy categories. 

\begin{ex}
The derived category $\cD_R$ of a ring $R$ is compactly generated by thick subcategory of perfect $R$-modules $\Perf_R=\cD_R^{\omega}$. By definition, the generating hypothesis holds for $R$ if the homology functor
\[
\xymatrix{H_*\colon \Perf_R \ar[r] & \Mod_R}
\] 
is faithful. In~\cite{hlpgenhypdermod}, Hovey, Lockridge, and Puninski provide a complete classification of those rings which satisfy the generating hypothesis. Indeed, the generating hypothesis holds for $R$ if and only if the following two conditions are satisfied:
\begin{enumerate}
	\item The ring $R$ has flat dimension at most 1.
	\item Every short exact sequence of finitely presented $R$-modules splits.  
\end{enumerate}
In particular, in this case $R$ is a von Neumann ring, but the converse does not hold. 
\end{ex}

\begin{ex}
Suppose $G$ is a finite group and $k$ is a field of characteristic $p$ dividing the order of $G$. The stable module category $\StMod_{kG}$ over the group algebra $kG$ is compactly generated by the simple $kG$-modules, and the homotopy groups with respect to shifts of the tensor unit $k$ can be identified with Tate cohomology. Building on earlier joint work with Benson and Christensen~\cite{bcc_genhyppgroup}, Carlson, Chebolu, and Min{\'a}{\v{c}} prove in \cite{ccmgenhypstmod} that the generating hypothesis holds in $\StMod_{kG}^{\omega}$ if and only if a $p$-Sylow subgroup of $G$ is cyclic of order 2 or 3. 
\end{ex}

Before we can state our criterion for the existence of counterexamples to the generating hypothesis, we need to introduce one more piece of terminology; note, however, that this notation is non-standard and will only be used in the remainder of this section. 

\begin{defn}
An object $X \in \cC$ is said to be non-projective if it is not a retract of a direct sum of shifts of $S$. 
\end{defn}

The next argument generalizes the proof of the main theorem of \cite{ccmgenhypstmod}.

\begin{prop}\label{prop:argenhyp}
Suppose $(\cC,S)$ is a compactly generated stable homotopy category, and let $\cD \subseteq \cC$ be a thick subcategory that admits Auslander--Reiten sequences. If $\cD$ contains non-projective indecomposables, then the generating hypothesis does not hold in $\cD$.
\end{prop}
\begin{proof} 
Choose an indecomposable object $Y$ in $\cD$ which is not a retract of a direct sum of shifts of $S$. This implies that $Y$ fits into an Auslander--Reiten sequence
\begin{equation}\label{eq:arseq}
\xymatrix{\Sigma^{-1}Z \ar[r] & X \ar[r] & Y \ar[r]^-{\delta} & Z}
\end{equation}
for some $X,Z \in \cD$. We claim that the morphism $\delta$ falsifies the generating hypothesis in $\cD$. Indeed, by definition, the sequence \eqref{eq:arseq} is not split and thus $\delta \ne 0$. It remains to show that $\pi_*\delta = 0$, so let $\Sigma^m S \to Y$ be some map with $m \in \Z$; note that $f$ cannot be a split epimorphism by assumption on $Y$. It follows that there is a lift 
\[
\xymatrix{& & \Sigma^m S \ar[d]^f \ar@{-->}[dl] \\ 
\Sigma^{-1}Z \ar[r] & X \ar[r] & Y \ar[r]^-{\delta} & Z,}
\]
forcing $\delta \circ f =0$. 
\end{proof}

\begin{rem}
As the proof shows, the existence of a sequence \eqref{eq:arseq} which satisfies conditions (1) and (2) of \Cref{defn:arsequence} suffices. However, in practice, this does not add much extra flexibility. 
\end{rem}

We conclude this section with an extended example demonstrating how \Cref{prop:argenhyp} can be used to construct counterexamples to the generating hypothesis for specific categories. 

\begin{ex}\label{ex:qcohgh}
Let $X$ be an $n$-dimensional smooth projective scheme over a field $k$ and let $\cD(X)$ be the derived category of quasi-coherent sheaves on $X$. This category is compactly generated by the bounded derived category $\cD^b(\Coh_X)$ of coherent sheaves on $X$, and $\pi_*\cF = H^{-*}\cF$ for all $\cF \in \cD(X)$. Serre duality provides a factorization
\[
\xymatrixcolsep{6pc}
\xymatrix{\cD^b(\Coh_X) \ar[r]^-{T(-) = -\otimes_{\cO_X}\Sigma^n\omega_X} \ar@{-->}[d]_-{\sim} & \cD(X)\\
\cD^b(\Coh_X), \ar@{^{(}->}[ru]}
\]
where $\omega_X \cong \Omega_{X/k}^n$ is the dualizing sheaf on $X$. As proven in \cite[Ex.~3]{krauseartzimmermann}, the functor $T$ is the Auslander--Reiten translation for $\cD(X)$, so $\cD(X)^{\omega} \simeq \cD^b(\Coh_X)$ admits Auslander--Reiten sequences by \Cref{thm:art}; see also \cite{jorgensenartschemes}. 

Note that, for non-trivial $X$, the category $\cD^b(\Coh_X)$ is not semi-simple, i.e., not every object is a retract of a direct sum of shifts of the structure sheaf $\cO_X$; for example, $\cO(1)$ is an indecomposable non-projective object. By Serre's finiteness theorem, $\Hom_X(\cF,\Sigma^m\cG) \cong H^m(\cF^{\vee} \otimes \cG)$ is finitely generated for all $\cF, \cG \in \cD^b(\Coh_X)$ and all $m \in \Z$, so $\cD^b(\Coh_X)$ is a Krull--Schmidt category by \cite[Cor.~9.9]{beligiannisart}. 

Since $\cD^b(\Coh_X)$ is not semi-simple, there exists an object $\cF \in \cD^b(\Coh_X)$ which is not a retract of direct sum of shifts of $\cO_X$; without loss of generality, we may assume that $\cF$ is indecomposable. Applying \Cref{prop:argenhyp} to $\cF$ then produces a map 
\[
\xymatrix{\delta\colon \cF \ar[r] & T(\cF)}
\]
that contradicts the generating hypothesis for $\cD^b(\Coh_X)$. Consequently, the generating hypothesis holds for $\cD^b(\Coh_X)$ if and only if $\cD^b(\Coh_X)$ is semi-simple, i.e., if and only if $X$ is 0-dimensional. %$X = \Spec(k^m)$ for some $m\ge 1$. 
Alternatively, this result can also be deduced from \cite[Prop.~5.2]{hgenhypdermod}. 
\end{ex}

\begin{rem}
In \cite[Ex.~3.12]{hgenhypdermod}, Lockridge studies the generating hypothesis for $\Thick(\cO_X)$, where $X$ is a finite-dimensional Noetherian scheme with enough locally free sheaves. In this case, the generating hypothesis holds if and only if $H^*(X,\cO_X)$ is a finite product of fields, e.g., for $X$ a smooth Fano variety. 
\end{rem}

\section{The $K(n)$-local generating hypothesis}\label{sec:kngh}

In this section, we work in the $p$-local stable homotopy category $\Sp$, for a fixed prime $p$. Let $n\ge 0$ be an integer and let $K(n)$ and $E_n$ denote Morava $K$-theory and Morava $E$-theory at height $n$, respectively. The $K(n)$-local stable homotopy category $\Sp_{K(n)}$ arises as the Bousfield localization of $\Sp$ with respect to $K(n)$. The $K(n)$-localization of the usual smash product induces a symmetric monoidal structure $\otimes$ on $\Sp_{K(n)}$ with unit $L_{K(n)}S^0$. Specializing the abstract generating hypothesis considered in \Cref{ssec:abstractgh} to this setting, we obtain:

\begin{kngh}
If $f\colon X \to Y$ is a morphism in $\Sp_{K(n)}^{\omega}$ with $\pi_*f = 0$, then $f=0$. 
\end{kngh}

Rational homotopy theory implies that $\Sp_{K(0)} \simeq \Sp_{\Q}$ is equivalent to the category of graded rational vector spaces, so the $K(n)$-local generating hypothesis holds trivially at height 0. The goal of this section is to systematically construct counterexamples for all heights $n>0$ and all primes $p$. 

\begin{rem}
The status \cite{mrstelescopeconj} of the telescope conjecture required a modification of Devinatz's and Hopkins' original approach to the generating hypothesis. Assuming a weak form of the chromatic splitting conjecture, Devinatz formulates a series of conjectures about the generating hypothesis in the $E_n$-local and $K(n)$-local stable homotopy category. These are related to the $K(n)$-local generating hypothesis, but it appears that our results do not affect his conjectures. 
\end{rem}

\subsection{Recollections on the $K(n)$-local stable homotopy category}

The $K(n)$-local stable homotopy category $\Sp_{K(n)}$ is a stable homotopy category with a number of interesting properties, which we review here for the convenience of the reader; our main reference is \cite{hovey_morava_1999}. 

The category $\Sp_{K(n)}$ is compactly generated by $F_n=L_{K(n)}F(n)$, the $K(n)$-localization of any finite type $n$ spectrum $F(n)$. It follows that
\[
\Sp_{K(n)}^{\omega} = \Thick_{\Sp_{K(n)}}(F_n).
\]
An object $X \in \Sp_{K(n)}$ is compact if and only if $E_n^*X$ is degreewise finite; if this is the case, then $\pi_*X$ is also degreewise finite. Note that the tensor unit $L_{K(n)}S^0$ is dualizable, but not compact unless $n=0$. In fact, an object $X$ is dualizable in $\Sp_{K(n)}$ if and only if $K(n)_*X$ is finite, and we will write $\Sp_{K(n)}^{\dual} \subset \Sp_{K(n)}$ for the thick subcategory of dualizable objects. Interestingly, $K(n)$-locally dualizable objects satisfy the Krull--Schmidt theorem, as shown in~\cite[Prop.~12.13]{hovey_morava_1999}.

\begin{prop}[Hovey--Strickland]\label{prop:knkrullschmidt}
The Krull--Schmidt theorem holds in $\Sp_{K(n)}^{\dual}$, i.e., any dualizable $K(n)$-local spectrum $X$ admits a decomposition 
\[
X \simeq \bigoplus_{Y\, \mathrm{indecomp}}Y^{n_{Y}}
\]
into a direct sum of indecomposable dualizable objects, for unique integers $n_Y \ge 0$ almost all of which are $0$. 
\end{prop}

Moreover, Brown's representability theorem holds for $\Sp_{K(n)}$, so the Brown--Comenetz duality functor $I$ can be lifted from $\Sp$ to a functor
\[
\xymatrix{\hat{I}\colon \Sp_{K(n)}^{\op} \ar[r] & \Sp_{K(n)}.}
\]
More explicitly, for any $X \in \Sp_{K(n)}$, these functors are related by the formula
\[
\hat{I}X \simeq F(M_nX,I),\] 
where $F$ denotes the internal mapping object of $\Sp$ and $M_nX$ is the $n$-th monochromatic layer of $X$. The properties of the local Brown--Comenetz duality functor $\hat{I}$ we need are summarized in the following result, proven in \cite[Thm.~10.2]{hovey_morava_1999}.

\begin{prop}[Hovey--Strickland]\label{prop:knbrowncomenetz}
There exists a cohomological functor $\hat{I}\colon \Sp_{K(n)}^{\op} \to \Sp_{K(n)}$ which is characterized by 
\begin{equation}\label{eq:bcduality}
\Hom_{\Z_p}(\pi_*(M_nX\otimes Y),\Q_p/\Z_p) \cong \Hom_{\Sp_{K(n)}}^{-*}(Y,\hat{I}X) 
\end{equation}
for all $X,Y \in \Sp_{K(n)}$. Furthermore, $\hat{I}$ restricts to a duality (contravariant equivalence) on the full subcategory $\Sp_{K(n)}^{\omega}$ of compact objects as well as the full subcategory $\Sp_{K(n)}^{\dual}$ of dualizable spectra. 
\end{prop}

\subsection{Auslander--Reiten sequences in $\Sp_{K(n)}$}

Employing the methods developed in \Cref{sec:abstract}, we now construct a version of Auslander--Reiten sequences for the $K(n)$-local category, and use them to produce counterexamples to the $K(n)$-local generating hypothesis. 

\begin{thm}\label{thm:localart}
The category $\Sp_{K(n)}^{\omega}$ of compact $K(n)$-local spectra admits Auslander--Reiten sequences for indecomposable objects $Y$:
\begin{equation}\label{eq:knarseq}
\xymatrix{\Sigma^{-1}TY \ar[r] & X \ar[r] & Y \ar[r]^-{\delta} & TY.}
\end{equation}
Moreover, the Auslander--Reiten translation $T$ is given by the dual of the Brown--Comenetz duality functor, i.e., $T \simeq \hat{I}\circ(-)^{\vee}$. 
\end{thm}
\begin{proof}
We have to check that the conditions of \Cref{thm:art} hold for $\Sp_{K(n)}$, and may assume $n>0$. We have already observed that $\Sp_{K(n)}$ is compactly generated by $F_n=L_{K(n)}F(n)$ for any finite type $n$ complex $F(n)$. Since the homotopy groups of $L_{K(n)}S^0$ are $p$-complete, the natural map $S^0 \to L_{K(n)}S^0$ induces a ring homomorphism $\Z_p \to \End_{\Sp_{K(n)}}(L_{K(n)}S^0)$. Hence $\Sp_{K(n)}$ is $\Z_p$-linear. Furthermore, for any pair of compact objects $X$ and $Y$, the  group
\[
\Hom_{\Sp_{K(n)}}(X,Y) \cong \pi_0(X^{\vee} \otimes Y)
\]
is finite and thus in particular finitely generated over $\Z_p$. 

It remains to identify the Auslander--Reiten translation for $\Sp_{K(n)}$ and to show that it factors through the canonical inclusion $\Sp_{K(n)}^{\omega} \subset \Sp_{K(n)}$. The injective hull of $k=\Z_p/p$ is $\Q_p/\Z_p$, so the Auslander--Reiten formula \eqref{eq:arduality} becomes
\begin{equation}\label{eq:arbcduality}
\Hom_{\Z_p}(\pi_0(X^{\vee}\otimes Y),\Q_p/\Z_p) \cong \Hom_{\Sp_{K(n)}}(Y,TX)
\end{equation}
for $X \in \Sp_{K(n)}^{\omega}$ and $Y \in \Sp_{K(n)}$. Because $X$ is a retract of the $K(n)$-localization of a finite type $n$ spectrum, the cofiber sequence
\[
\xymatrix{M_nX \ar[r] & L_{n}X \ar[r] & L_{n-1}X \simeq 0}
\]
shows that $M_nX \simeq X$ and similarly for $X^{\vee}$; therefore:
\[
\Hom_{\Z_p}(\pi_0(M_nX\otimes Y),\Q_p/\Z_p) \cong \Hom_{\Z_p}(\pi_0(X\otimes Y),\Q_p/\Z_p).
\]
Replacing $X$ by $X^{\vee}$ in \eqref{eq:bcduality} and comparing this to \eqref{eq:arbcduality}, we see that $TX$ and $\hat{I}(X^{\vee})$ represent the same functor, so the Yoneda lemma implies $TX \simeq \hat{I}(X^{\vee})$ for all $X \in \Sp_{K(n)}^{\omega}$. Finally, \Cref{prop:knbrowncomenetz} provides a factorization
\[
\xymatrix{\Sp_{K(n)}^{\omega} \ar[r]^-{(-)^{\vee}}_-{\sim} \ar@{-->}[rd]_{\sim} & \Sp_{K(n)}^{\omega,\op} \ar[r]^-{\hat{I}} & \Sp_{K(n)} \\
& \Sp_{K(n)}^{\omega}, \ar@{^{(}->}[ru]_{\iota}}
\]
thereby finishing the proof. 
\end{proof}

\begin{rem}
Gross--Hopkins duality \cite{grosshopkinseqvb,grosshopkinsannouncement} identifies the Auslander--Reiten translation of $\Sp_{K(n)}$ more explicitly as
\[
TF \simeq \hat{I}DF \simeq F \otimes \hat{I}S^0
\]
for all $F \in \Sp_{K(n)}^{\omega}$, where $\hat{I}S^0$ is a $K(n)$-locally dualizable spectrum. 
\end{rem}

\begin{cor}\label{cor:localgenhyp}
The $K(n)$-local generating hypothesis does not hold for any $n>0$ and any prime $p$.
\end{cor}
\begin{proof}
By \Cref{thm:localart} and \Cref{prop:argenhyp}, it suffices to show that $\Sp_{K(n)}^{\omega}$ contains an indecomposable object which is not equivalent to a retract of a shift of $L_{K(n)}S^0$. First note that, if $f\colon L_{K(n)}S^0 \to X$ is a split epimorphism, then $K(n)_*f$ is either $0$ or an isomorphism, because $K(n)_*L_{K(n)}S^0 = K(n)_*$ is a graded field. Therefore, $X \simeq L_{K(n)}S^0$ or $X\simeq 0$, hence $L_{K(n)}S^0$ is indecomposable. 

Let $Y \in \Sp_{K(n)}^{\omega}$ be any non-trivial object; without loss of generality, we may assume that $Y$ is indecomposable. Since $L_{K(n)}S^0$ is indecomposable but not compact in $\Sp_{K(n)}$ whenever $n>0$, $Y$ cannot be a retract of a shift of $L_{K(n)}S^0$, so the map $\delta$ in the Auslander--Reiten sequence \eqref{eq:knarseq} is a counterexample to the $K(n)$-local generating hypothesis.  
\end{proof}

This implies that the $E_n$-local generating hypothesis holds if and only if $n=0$. In other words, writing $\Sp_{E_n}$ for the $E_n$-local stable homotopy category, we have:

\begin{cor}
For any $p$, the functor $\pi_*\colon \Sp_{E_n}^{\omega} \to \Mod_{\Z_{(p)}}$ is faithful if and only if $n=0$. 
\end{cor}
\begin{proof}
Since $\Sp_{K(n)}^{\omega} \subseteq \Sp_{E_n}^{\omega}$ and 
\[
\Hom_{\Sp_{E_n}}^*(L_nS^0,X) \cong \Hom_{\Sp_{K(n)}}^*(L_{K(n)}S^0,X)\] 
for $X \in \Sp_{K(n)}^{\omega}$, the result follows immediately from \Cref{cor:localgenhyp}.
\end{proof}

\begin{rem}
In \cite[Sec.~4]{devinatzgh2}, Devinatz explicitly constructs non-trivial homotopy classes $h_I \in \pi_*L_{K(n)}M_I$ for any $n \ge 2$ and sufficiently large primes $p$, where $M_I$ denotes a generalized Moore spectrum of type $n$. He then shows in his Appendix 2 that the Spanier--Whitehead dual of $h_I$,
\[
\xymatrix{h_I^{\vee}\colon \Sigma^mL_nDM_I \ar[r] & L_nS^0,}
\]
satisfies $\pi_*h_I^{\vee}=0$, thereby also providing a counterexample to the $E_n$-local generating hypothesis. The advantage of his construction is that the target of $h_I^{\vee}$ is $L_nS^0$, but these classes do not immediately give counterexamples to the $K(n)$-local generating hypothesis. 
\end{rem}

One feature of the $K(n)$-local stable homotopy category is that it has an interesting Picard group $\Pic_n$ of invertible objects with respect to the localized smash product. Arguably, it is thus more natural to grade homotopy groups over $\Pic_n$, i.e., to consider homotopy groups 
\[
\pi^P(X) = \Hom_{\Sp_{K(n)}}(P,X)
\]
for any $P \in \Pic_n$, see \cite[Sec.~14]{hovey_morava_1999}. Consequently, we can formulate the $\Pic$-graded $K(n)$-local generating hypothesis.

\begin{pickngh}
If $f\colon X \to Y$ is a map in $\Sp_{K(n)}^{\omega}$ such that $\pi^P(f) = 0$ for all $P \in \Pic_n$, then $f=0$. 
\end{pickngh}

Although this statement is a priori weaker than the $K(n)$-local generating hypothesis, the result remains the same. 

\begin{cor}\label{cor:piclocalgenhyp}
The $\Pic$-graded $K(n)$-local generating hypothesis does not hold for any $n>0$ and any prime $p$.
\end{cor}
\begin{proof}
An object $P \in \Sp_{K(n)}$ is invertible if and only if $E_n^*P \cong \Sigma^mE_n^*$ for some $m \in \Z$, by \cite[Prop.~14.2]{hovey_morava_1999}, so $P$ cannot be compact. This also implies that $K(n)_*P \cong \Sigma^mK(n)_*$ for $P\in \Pic_n$, whence $P$ is indecomposable. Therefore, the same argument as in \Cref{cor:localgenhyp} yields the result.
\end{proof}

\begin{rem}
A similar proof shows that the $\Pic$-graded version of the $E_n$-local generating hypothesis holds if and only if $n=0$. 
\end{rem}

\begin{rem}
After discovering the above proof, we noticed that, at height $n=1$ and $p>2$, a different counterexample to the ($\Pic$-graded) $K(n)$-local generating hypothesis is given in \cite[Sec.~15.1]{hovey_morava_1999}, where it is attributed to Mike Hopkins. Presumably, the example given there could be generalized to higher heights using the Adams operations on Morava $E$-theory constructed by Ando in \cite[Prop.~3.6.1]{andoisog}, but we have not checked the details.
\end{rem}

\bibliography{bibliography}
\bibliographystyle{alpha}
\end{document}